\documentclass[12pt]{amsart}

\usepackage{amsthm}
\newtheorem{theorem}{Theorem}[section]
\newtheorem{lemma}[theorem]{Lemma}

\theoremstyle{definition}

\theoremstyle{remark}

\counterwithin*{section}{part}

\usepackage{amssymb}
\usepackage{empheq}
\usepackage{stmaryrd}
\usepackage{enumerate}
\usepackage[shortlabels]{enumitem}
\usepackage{calc}
\usepackage{url}
\usepackage{comment}

\newcommand{\CAT}{{\rm{CAT}(0)}}

\def\EE{\mathbb{E}}
\def\PP{\mathbb{P}}

\usepackage[left=2.0cm,%
                right=2.0cm,%
                top=2.5cm,%
                bottom=3.5cm,%
                headheight=12pt,%
                a4paper]{geometry}%

\begin{document}

\title[Weak convergence of the stochastic proximal point method]{Weak convergence of the stochastic proximal point method in metric spaces}

\author[N. Pischke]{Nicholas Pischke}
\date{\today}
\maketitle
\vspace*{-5mm}
\begin{center}
{\scriptsize 
Department of Computer Science, University of Bath,\\
Claverton Down, Bath, BA2 7AY, United Kingdom.\\
E-mail: nnp39@bath.ac.uk}
\end{center}

\maketitle
\begin{abstract}
We prove the almost sure weak convergence of a stochastic proximal point method for minimizing a convex integral function in the general nonlinear context of complete geodesic metric spaces of nonpositive curvature (so-called Hadamard spaces), solving a problem of M. Ba\v{c}\'ak. This method, formulated in the context of a mild growth condition on the function which generalizes Lipschitz continuity, was previously only considered in the context of strong metric regularity conditions or in the context of locally compact spaces (all of which immediately induce strong convergence). In particular, this result is novel already over Hilbert spaces. The proof is a combination of a weak almost sure convergence theorem for stochastic processes in Hadamard spaces which confine to a stochastic variant of quasi-Fej\'er monotonicity, due to previous work of the author, together with a new argument for proving the almost sure convergence of the mean function values of the process towards the minimal value.
\end{abstract}


\section{Introduction}

\subsection{Background and motivation}
The problem of minimizing a convex integral function, that is solving the problem
\[
\min_{x\in X}\int f(e,x)\,d\mu(e),
\]
for a given normal convex integrand $f:E\times X\to (-\infty,+\infty]$ (see \cite{Rockafellar1971}) on a complete probability space $(E,\mathcal{E},\mu)$ and some target Hilbert space $X$, is one of the most important general formulations of stochastic approximation. Motivated by the seminal proximal point algorithm for approximating minimizers of ``ordinary'' convex functions on such spaces, which goes back to the work of Rockafellar \cite{Rockafellar1976}, Martinet \cite{Martinet1970} as well as Br\'ezis and Lions \cite{BrezisLions1978}, one of the prevalent modern tools for approaching this problem is the stochastic proximal point method
\[
x_{n+1}:=\mathrm{argmin}_{y\in X}\left\{ f(\xi_{n+1},y)+\frac{1}{2\lambda_n}\lVert x_n-y\rVert^2\right\},
\]
formulated over an auxiliary probability space $(\Omega,\mathcal{F},\PP)$, given a starting point $x_0\in X$, a sequence of parameters $(\lambda_n)\subseteq (0,\infty)$ with certain growth conditions and a sequence $(\xi_{n+1})$ of random variables $\xi_{n+1}:\Omega\to E$ which are independent and identically distributed (i.i.d.)\ with (common) distribution $\mu$. 

The basic concepts of this method go back at least to the work of Bertsekas \cite{Bertsekas2011} where it was formulated in the context of finite probability spaces $(E,\mathcal{E},\mu)$ and over Euclidean spaces $X=\mathbb{R}^d$, establishing the convergence of the method under suitable assumptions on the surrounding parameters, in particular including the assumption that any $f(e,\cdot)$ is $L$-Lipschitz for a fixed constant $L>0$. The method and corresponding convergence results are further extended in \cite{WangBertsekas2013,WangBertsekas2015}, to the above general problem formulation and beyond. Since then, this iteration and its variants were widely studied and we refer to \cite{AsiChadhaChengDuchi2020,AsiDuchi2019,Bertsekas2012,Bianchi2016,EisenmannStillfjordWilliamson2022,NemirovskiJuditskyLanShapiro2009,RyuBoyd}, among many others, for various such discussions (in particular regarding their complexities). 

Extensions of these tools from (stochastic) convex analysis to nonlinear settings are of high practical relevance, particularly because of the extensive developments of machine learning in recent years, where optimization over nonlinear spaces such as manifolds plays a key role (see e.g.\ \cite{ZhangSra2016}). However, the need for suitable methods in nonlinear contexts is not only driven through optimization on manifolds or other domains with differentiable structure, but also in particular by spaces with much sparser structure naturally occurring in applications, such as e.g.\ the Billera-Holmes-Vogtmann tree space \cite{BilleraHolmesVogtmann2001} prominently used in phylogenetics (see also e.g.\ \cite{Bacak2014b}).

In this paper, we study the stochastic proximal point method in the context of the general class of geodesic metric spaces with nonpositive curvature, as introduced in the work of Alexandrov \cite{Aleksandrov1951}. Also known as $\CAT$ spaces, following the work of Gromov \cite{Gromov1987}, these spaces uniformly cover spaces such as Hilbert spaces, $\mathbb{R}$-trees and Hadamard manifolds (i.e.\ complete simply connected Riemannian manifolds of nonpositive sectional curvature) as well as the previously mentioned Billera-Holmes-Vogtmann tree space or the Hilbert ball, and further involved examples. Authoritative references for geodesic and $\CAT$ spaces are in particular the works \cite{AlexanderKapovitchPetrunin2023,BridsonHaefliger1999} as well as \cite{Bacak2014a}, with the latter providing a shorter treatment focused in particular on convex analysis and optimization.

The deterministic proximal point method was lifted to the setting of Hadamard spaces (that is \emph{complete} geodesic metric spaces with nonpositive curvature) by Ba\v{c}\'ak \cite{Bacak2013}, relying on metric variants of the proximal mappings. In particular, the work \cite{Bacak2013} established weak convergence in Hadamard spaces of this deterministic proximal point method. Naturally, by G\"uler's seminal work \cite{Gueler1991}, this is the most one can hope for already in Hilbert spaces.

The stochastic proximal point method, as sketched over Hilbert spaces above, was extended to the setting of (separable) Hadamard spaces in the work \cite{Bacak2018}, building on preceding work \cite{Bacak2014b} on a splitting proximal point method with random order for finite sums of convex functions over similar spaces.

\subsection{Main results and related work}

We first fix the formal framework. In similarity to the above, let $(\Omega,\mathcal{F},\PP)$ and $(E,\mathcal{E},\mu)$ be probability spaces, with $(E,\mathcal{E},\mu)$ complete, and let $X$ now be a separable Hadamard space. In analogy to \cite{Rockafellar1971} (see also \cite{CastaingValadier1977}), let $f:E\times X\to (-\infty,+\infty]$ be a normal convex integrand, i.e.\ $f(e,\cdot)$ is proper, lower-semicontinuous (lsc) and convex\footnote{Given a Hadamard space $X$, recall that a function $f:X\to (-\infty,+\infty]$ is called lsc if $\liminf_{n\to\infty} f(x_n)\geq f(x)$ whenever $x_n\to x$ in $X$, and convex if $f\circ \gamma$ is convex for any geodesic in $X$.} for all $e\in E$ and $f$ is $\mathcal{E}\otimes\mathcal{B}(X)$-measurable. Define the proximal map of $f$ via
\[
\mathrm{prox}_{\lambda}^f(e,x):=\mathrm{argmin}_{y\in X}\left\{ f(e,y)+\frac{1}{2\lambda}d^2(x,y)\right\},
\]
which is well-defined for all $e\in E$, $x\in X$ and $\lambda>0$ (see e.g.\ \cite{Jost1995,Mayer1998}). Further, $\mathrm{prox}_{\lambda}^f(e,\cdot)$ is nonexpansive for any $e\in E$ and $\lambda>0$ (see e.g.\ Lemma 4 in \cite{Jost1995}), and also $\mathrm{prox}_{\lambda}^f(\cdot,x)$ is measurable for any $x\in X$ and $\lambda>0$. Hence, $\mathrm{prox}_{\lambda}^f$ is a Carath\'eodory function and so in particular $\mathcal{E}\otimes\mathcal{B}(X)$-measurable (see e.g.\ Lemma 8.2.6 in \cite{AubinFrankowska2009}).

The stochastic proximal point method is then given by the iteration
\[
x_{n+1}:=\mathrm{prox}^f_{\lambda_n}(\xi_{n+1},x_n),\tag{SPPA}\label{RPPA}
\]
given, as before, a starting point $x_0\in X$ and sequences $(\lambda_n)$ of positive reals as well as $(\xi_{n+1})$ of variables $\Omega\to E$, for which we assume that 
\[
(\xi_{n+1})\text{ is i.i.d.\ with distribution $\mu$ and }\sum_{n\in\mathbb{N}}\lambda_n=+\infty, \sum_{n\in\mathbb{N}}\lambda_n^2<+\infty.\tag{A1}\label{parameters}
\]
The work \cite{Bacak2018} in particular relies on a certain weak growth condition on the integrand introduced therein, which is a generalization of many of the common growth conditions from the literature: Assume there exists a positive function $L\in L^2(E,\mu)$ and a point $p\in X$ such that
\[
f(e,x)-f(e,y) \leq L(e)(1+d(x,p))d(x,y)\tag{A2}\label{growth}
\]
for all $x,y\in X$ and almost all $e\in E$. We refer to \cite{Bacak2018} for further discussions of this condition as well as on the benefits of the proximal point method at large compared to methods such as gradient descent (see also e.g.\ \cite{Bertsekas2011}), especially in contexts such as Hadamard spaces, where usual stochastic gradient methods cannot be used without additional differential structure that one would normally have access to on e.g.\ manifolds. Beyond that, we refer to the excellent exposition in \cite{Bacak2013,Bacak2014b,Bacak2018} for further discussions on these methods and relations other works.

The only previous \emph{general} convergence result on the stochastic proximal point method in Hadamard spaces formulated above, that is \emph{without} any further regularity assumptions on $f$ or its mean such as strong convexity or weak sharp minima,\footnote{We refer e.g.\ to the recent work \cite{PischkePowell2026} which provides quantitative results for the proximal point method considered herein under stronger regularity assumptions, such as weak sharp minima (and beyond), relying on rather general considerations on stochastic quasi-Fej\'er monotonicity in a metric context (see also the related \cite{NeriPischkePowell2025}).} is a strong convergence result in the context of a local compactness assumption, given in \cite{Bacak2018} (see also \cite{Bacak2023}). 

\begin{theorem}[Theorem 3.1 in \cite{Bacak2018}]\label{RPPAstrongConv}
Let $(E,\mathcal{E},\mu)$ and $(\Omega,\mathcal{F},\PP)$ be probability spaces, with $(E,\mathcal{E},\mu)$ complete, and let $X$ be a locally compact Hadamard space. Let $f:E\times X\to (-\infty,+\infty]$ be a normal convex integrand such that $F(x):=\int f(e,x)\,d\mu(e)$ is proper and $\mathrm{argmin}F\neq\emptyset$. Let $(x_n)$ be the iteration given by \eqref{RPPA}, and assume \eqref{parameters} as well as \eqref{growth}.

Then $(x_n)$ a.s.\ strongly converges to an $\mathrm{argmin}F$-valued random variable.
\end{theorem}

In this paper, we prove the weak convergence of the stochastic proximal point method as formulated in \cite{Bacak2018} in the context of general (separable) Hadamard spaces. 

\begin{theorem}\label{RPPAweakConv}
Let $(E,\mathcal{E},\mu)$ and $(\Omega,\mathcal{F},\PP)$ be probability spaces, with $(E,\mathcal{E},\mu)$ complete, and let $X$ be a separable Hadamard space. Let $f:E\times X\to (-\infty,+\infty]$ be a normal convex integrand such that $F(x):=\int f(e,x)\,d\mu(e)$ is proper and $\mathrm{argmin}F\neq\emptyset$. Let $(x_n)$ be the iteration given by \eqref{RPPA}, and assume \eqref{parameters} as well as \eqref{growth}.

Then $(x_n)$ a.s.\ weakly converges to an $\mathrm{argmin}F$-valued random variable.
\end{theorem}

This in particular solves Problem 6.6 of Ba\v{c}\'ak \cite{Bacak2023}. Again, by the results of G\"uler \cite{Gueler1991}, this is best possible already in Hilbert spaces.

Further, as the stochastic proximal point method generalizes the splitting proximal point method with random order given in \cite{Bacak2014b} (see Theorem 3.7 therein), we also get the following solution to Problem 6.4 of Ba\v{c}\'ak \cite{Bacak2023}:\footnote{The choice to allow general $x, y \in X$ in \eqref{growth} is purely for simplicity. Throughout the paper, it would suffice to assume this condition only along the sequence $(x_n)$, akin to \cite{Bacak2014b}. In that context, it is then immediate to see that the assumption \eqref{growth} in Theorem \ref{SplitPPAweak} generalizes the growth condition used in Theorem 3.7 in \cite{Bacak2014b}. We however stay with this formulation of \eqref{growth} here, for simplicity.}

\begin{theorem}\label{SplitPPAweak}
Let $(\Omega,\mathcal{F},\PP)$ be a probability space and let $X$ be a separable Hadamard space. Let $F(x):=\sum_{k=1}^Nf_k(x)$, where each $f_k:X\to (-\infty,+\infty]$ is proper convex lsc, and where we assume that $\mathrm{argmin}F\neq\emptyset$. Let $(x_n)$ be the iteration given by 
\[
x_{n+1}:=\mathrm{prox}^{f_{r_{n+1}}}_{\lambda_n}(x_n)
\]
where $(r_{n+1})$ is a sequence of independent random variables which attain values in $\{1,\dots,N\}$ according to the uniform distribution. Assume \eqref{parameters} and \eqref{growth}.\footnote{\label{splitPPAfootnote}By \eqref{growth}, we here mean the condition that $f_k(x) - f_k(y) \leq L(k)(1 + d(x, p))d(x, y)$ for all $k = 1, \dots , N$ and $x, y \in X$, for some $L(k) > 0$, see also the proof of Theorem \ref{SplitPPAweak} given later on.}

Then $(x_n)$ a.s.\ weakly converges to an $\mathrm{argmin}F$-valued random variable.
\end{theorem}

It should be noted that according to \cite{Bacak2023}, both results are new even in Hilbert spaces which seems to be true, to our knowledge. Indeed, at the level of general separable Hilbert spaces $X$ without any further regularity assumptions, the only other result on the method \eqref{RPPA} that we are aware of derives from the work of Bianchi \cite{Bianchi2016}, which studies the broader stochastic proximal point algorithm induced by the iteration
\[
x_{n+1}:=J_{\lambda_n}(\xi_{n+1},x_n)
\]
of the resolvents $J_\lambda(e,x):=(\mathrm{Id}+\lambda A(e,\cdot))^{-1}(x)$ of general stochastically perturbed monotone set-valued operators $A:E\times X\to 2^X$, with $(\xi_{n+1})$ and $(\lambda_n)$ as before. For this method, under suitable assumptions (including \eqref{parameters}, next to additional uniform integrability requirements), the work of \cite{Bianchi2016} then establishes weak ergodic convergence of $(x_n)$, that is $\overline{x}_n:=\sum_{k=0}^n\lambda_kx_k/\sum_{k=0}^n\lambda_k$ converges weakly, to a zero of the mean operator $\underline{A}(x):=\int A(s,x)\,d\mu(s)$, where the integral refers to the Aumann integral \cite{Aumann1965}.\footnote{The work of Bianchi \cite{Bianchi2016} also discusses stronger assumptions on the operator $A$ under which almost sure strong convergence can be guaranteed. In particular, under the assumption of strong monotonicity, the results of Bianchi \cite{Bianchi2016} have recently been extended to randomly perturbed monotone vector fields on (suitable) Hadamard spaces in \cite{Pischke2025}. Whether the weak ergodic convergence of \cite{Bianchi2016} extends to this setting of \cite{Pischke2025} however remains open.} This applies to the above method \eqref{RPPA} by taking $A=\partial f$, with $\partial f$ being the (stochastic) subgradient of $f$ (see also the discussions in \cite{Bianchi2016}). It in fact follows from results of Brezis and Lions \cite{BrezisLions1978} that this result of Bianchi cannot be extended to (weak) almost sure convergence for general operators $A$ in Hilbert spaces. In that way, the present results in particular show that this obstacle can be overcome for $A=\partial f$, for suitable $f$.

Our argument is based on two key ingredients: The first is a general (weak) convergence theorem for stochastic processes in Hadamard spaces which confine to a stochastic variant of quasi-Fej\'er monotonicity (see Lemma \ref{weakConvergence} later on), derived in recent work of the author \cite{Pischke2026}. This result, which essentially represents an applied implementation of the seminal Robbins-Siegmund theorem on almost-supermartingale convergence \cite{RobbinsSiegmund1971} (on which also the proof crucially relies), extends similar such general stochastic convergence theorems obtained by Combettes and Pesquet \cite{CombettesPesquet2015} over Hilbert spaces to the setting of Hadamard spaces, and simultaneously extends related deterministic convergence theorems over Hadamard spaces obtained by Ba\v{c}\'ak, Searston and Sims \cite{BacakSearstonSims2012} to the stochastic setting. In particular, the result relies on a nonlinear variant of Pettis' theorem \cite{Pettis1938}, similarly derived in \cite{Pischke2026}. Over Hilbert spaces, this result from \cite{Pischke2026} is not necessary and one can fully rely on the stochastic convergence theorem of Combettes and Pesquet obtained in \cite{CombettesPesquet2015}.

The second key ingredient is a new convergence result for a certain stochastic recursive inequality, used for proving the almost sure convergence of the mean function values of the process towards the minimal value (see Lemma \ref{LipschitzSum} later on). This result can be considered as a stochastic extension of a similar deterministic result due to Alber, Iusem and Solodov \cite{AlberIusemSolodov1998} (discussed in more detail later, see also \cite{NeriPowell2023}), and in particular combines some of the arguments found there with Kolmogorov's two-series theorem (see Lemma \ref{Kolmogorov} later on). This in particular is related but ultimately slightly different from a similar stochastic variant of the result of Alber, Iusem and Solodov \cite{AlberIusemSolodov1998} derived by Geiersbach and Pflug \cite{GeiersbachPflug2019} (see Proposition 3.2 therein). That variant was used in \cite{GeiersbachPflug2019} to show the weak almost sure convergence of a projected stochastic gradient algorithm for convex and $L^2$-Fr\'echet differentiable objective functionals (with convex constraint sets). While this is similar in spirit to the present paper, the work \cite{GeiersbachPflug2019} not only focuses on a different algorithm but is in particular phrased for separable Hilbert spaces, and in that context relies on an additive representation of stochastic errors (as also manifested in the use of Proposition 3.2 therein), whereas our result is particularly tailored to use in a nonlinear setting such as that described above.

Even though we only consider the stochastic proximal point method over Hadamard spaces as given in \cite{Bacak2018} as well as its special case from \cite{Bacak2014b} in the present paper, we think that the general approach taken here will be of use for the convergence analysis of further methods from stochastic convex optimization already over Hilbert spaces, but in particular also over nonlinear spaces. 

Over Hilbert spaces, this could perhaps include variants of proximal-gradient algorithms as e.g.\ studied in  \cite{GeiersbachScarinci2021} (among many other references), where however contrary to the methods studied therein, both the gradient and the proximal step are stochastic. Such a type of method was recently studied by Madariaga \cite{Madariaga2026}, which however ultimately relied on a strong regularity condition on the functions (see Assumption 3.2 in \cite{Madariaga2026}) that, among other things, yields strong almost sure convergence.

On the nonlinear side, one particular example we want to mention here is the Busemann subgradient method introduced recently by Goodwin, Lewis, L\'opez-Acedo and Nicolae \cite{GoodwinLewisLopezAcedoNicolae2024}, and its extension for stochastic minimization as considered in \cite{Pischke2026}, as well as combinations of that method with a stochastic proximal step in the style of a proximal-gradient method similar to the above.

\section{Preliminaries}

We now discuss the few preliminary definitions, results and notations that we require throughout. As mentioned in the introduction, beyond the results indicated here, we refer to \cite{AlexanderKapovitchPetrunin2023,Bacak2014a,BridsonHaefliger1999} for a comprehensive overview of geodesic metric spaces and their properties, in particular to \cite{Bacak2014a} for aspects of (stochastic) optimization. Further, we refer to e.g.\ \cite{Klenke2020} for a standard textbook on probability theory.

Let $(X,d)$ be a metric space. A geodesic is an isometry $\gamma:[0,l]\to X$. We say that it joins $x=\gamma(0)$ and $y=\gamma(l)$ (where necessarily $l=d(x,y)$). $X$ is called (uniquely) geodesic is every two points are joined by a (unique) geodesic. A geodesic metric space $(X,d)$ is called a $\CAT$ space (also called a space of nonpositive curvature in the sense of Alexandrov) if it satisfies 
\[
d^2(\gamma(tl),x)\leq (1-t)d^2(\gamma(0),x)+td^2(\gamma(l),x)-t(1-t)d^2(\gamma(0),\gamma(l))
\]
for all $x\in X$, all $t\in [0,1]$ and all geodesics $\gamma:[0,l]\to X$ (that is, an extension of the so-called Bruhat-Tits $\mathrm{CN}$-inequality \cite{BruhatTits1972} to geodesics). Any $\CAT$ space is uniquely geodesic and a complete $\CAT$ space is called a Hadamard space.

Weak convergence in $\CAT$ spaces goes back to the work of Jost \cite{Jost1994} and is often called $\Delta$-convergence following the work of Kirk and Panyanak \cite{KirkPanyanak2008} (we refer in particular to the discussion in \cite{Bacak2013} on that matter). We define weak convergence here as follows (see e.g.\ \cite{Bacak2014a}):

Given a bounded sequence $(x_n)\subseteq X$ and a point $x\in X$, their asymptotic radius is given by
\[
r(x_n,x):=\limsup_{n\to\infty}d^2(x_n,x)
\]
and the general asymptotic radius of the sequence $(x_n)$ is given by
\[
r(x_n):=\inf_{x\in X}r(x_n,x).
\]
A point $x\in X$ is called an asymptotic center of $(x_n)$ if $r(x_n,x)=r(x_n)$. In Hadamard spaces, asymptotic centers exist and are unique (see e.g.\ Proposition 7 in \cite{DhompongsaKirkSims2006}).

We say that $(x_n)$ weakly converges to $x\in X$, written $x_n\to^w x$, if $x$ is the asymptotic center of each subsequence of $(x_n)$. A point $x\in X$ is a weak cluster point of $(x_n)$ if there is a subsequence $(x_{n_k})$ of $(x_n)$ with $x_{n_k} \to^w x$.

We write $\mathfrak{W}(x_n)$ for the set of all weak cluster points of $(x_n)$ and $\mathfrak{S}(x_n)$ for the set of all strong cluster points of $(x_n)$, defined as usual using the metric.

Throughout this paper, we now fix a separable Hadamard space $(X,d)$ and two probability spaces, $(\Omega,\mathcal{F},\PP)$ and $(E,\mathcal{E},\mu)$, with $(E,\mathcal{E},\mu)$ complete. All probabilistic notions such as measurability, random variables, almost sureness (a.s.), expectation, etc., are understood relative to the space $(\Omega,\mathcal{F},\PP)$, if not stated otherwise. In particular, an $X$-valued random variable is a map $x:\Omega\to X$ which is measurable relative to $\mathcal{F}$ and the Borel $\sigma$-algebra $\mathcal{B}(X)$ of that space. We denote (conditional) expectations over $(\Omega,\mathcal{F},\PP)$ by $\EE$. All properties as well as (in-)equalities between random variables are understood to hold only almost surely, if not stated otherwise.

\section{Main results}

\subsection{Key lemmas}

As outlined in the introduction, the first main technical ingredient we need is a general result on the weak convergence of stochastic quasi-Fej\'er monotone sequences in metric spaces.

For that, and for the rest of the paper, we use the following notation (similar to \cite{Pischke2026}): Let $\mathsf{F}=(\mathsf{F}_n)$ be a given filtration of $\mathcal{F}$, that is a sequence of sub-$\sigma$-algebras of $\mathcal{F}$ such that $\mathsf{F}_n\subseteq\mathsf{F}_m$ for $n\leq m$. We write $\ell_+(\mathsf{F})$ for the set of sequences of non-negative real-valued random variables $(e_n)$ that are adapted to the filtration, i.e.\ where $e_n$ is $\mathsf{F}_n$-measurable for all $n\in\mathbb{N}$. Further, we write $\ell^1_+(\mathsf{F})$ for the set of all $(e_n)\in \ell_+(\mathsf{F})$ such that $\sum_{n\in\mathbb{N}}e_n<+\infty$ a.s.

\begin{lemma}[Proposition 4.3 in \cite{Pischke2026}]\label{weakConvergence}
Let $X$ be a separable Hadamard space and let $Z\subseteq X$ be a nonempty closed subset of $X$ and let $\phi:[0,+\infty)\to [0,+\infty)$ be strictly increasing such that $\lim_{t\to +\infty}\phi(t)=+\infty$. Let $\mathsf{F}=(\mathsf{F}_n)$ be a filtration and let $(x_n)$ be a sequence of $X$-valued random variables adapted to $\mathsf{F}$ such that it is stochastically quasi-Fej\'er monotone w.r.t.\ $Z$, that is for any $z\in Z$ there are $(\chi_n(z)),(\eta_n(z))\in\ell^1_+(\mathsf{F})$ and $(\theta_n(z))\in\ell_+(\mathsf{F})$ such that for all $n\in\mathbb{N}$:
\[
\EE[\phi(d(x_{n+1},z))\mid \mathsf{F}_n]\leq (1+\chi_n(z)) \phi(d(x_n,z))-\theta_n(z)+\eta_n(z)\text{ a.s.}\tag{$*$}\label{randFejer}
\]
Then we have the following assertions:
\begin{enumerate}
\item $\sum_{n\in\mathbb{N}}\theta_n(z)<+\infty$ a.s.\ for all $z\in Z$.
\item $(x_n)$ is bounded a.s.
\item There exists a set $\widetilde{\Omega}$ with $\PP(\widetilde{\Omega})=1$ such that for all $\omega\in\widetilde{\Omega}$ and $z\in Z$, the sequence given by $d(x_n(\omega),z)$ converges.
\item If $\mathfrak{W}(x_n)\subseteq Z$ a.s., then $(x_n)$ weakly converges a.s.\ to a $Z$-valued random variable.
\item If $\mathfrak{S}(x_n)\cap Z\neq\emptyset$ a.s., then $(x_n)$ strongly converges a.s.\ to a $Z$-valued random variable.
\end{enumerate}
\end{lemma}

In fact, we later rely on a slightly strengthened variant of item (2) in Lemma \ref{weakConvergence} above, showing that $\sup_{n\in\mathbb{N}} d(x_n,z)<+\infty$ a.s.\ for some (any) $z\in Z$. Indeed, this follows immediately from item (3) above, since $d(x_n,z)$ converges a.s.\ for any $z\in Z$.

It is easy to see (see e.g.\ Section 4.1 in \cite{NeriPowell2025}) that for a general nonnegative stochastic process $(X_n)$, being almost surely uniformly bounded in the sense of $\sup_{n\in\mathbb{N}} X_n<+\infty$ a.s.\ is equivalent to the statement that for any $\lambda>0$, there exists an $N>0$ such that $\PP\left(\sup_{n\in\mathbb{N}}X_n> N\right)\leq\lambda$. Following \cite{NeriPowell2025}, we call a function $\psi:(0,1)\to (0,\infty)$ that provides such an $N$ in terms of $\lambda>0$ a modulus of uniform boundedness. The following lemma just collects the fact that such a modulus exists for $X_n:=d(x_n,z)$ in the setting of Lemma \ref{weakConvergence} for later use.

\begin{lemma}\label{unifBd}
In the setting of Lemma \ref{weakConvergence}, fix $z\in Z$. Then $\sup_{n\in\mathbb{N}} d(x_n,z)<+\infty$ a.s. In particular, there exists a function $\psi:(0,1)\to (0,\infty)$ such that
\[
\PP\left(\sup_{n\in\mathbb{N}} d(x_n,z)> \psi(\lambda)\right)\leq\lambda
\]
for any $\lambda\in (0,1)$.
\end{lemma}

In fact, such a modulus can be explicitly constructed from associated quantitative data witnessing the key assumptions of Lemma \ref{weakConvergence}, and details for such constructions can be found in an abstract setting in \cite{NeriPowell2026}.

The second main ingredient of the present work is a general result on the almost sure convergence of stochastic processes satisfying a Lipschitz-type condition controlled by an i.i.d.\ process.

For that, we rely on Kolmogorov's two-series theorem:

\begin{lemma}[Theorem 12.2 in \cite{Williams1991}]\label{Kolmogorov}
Let $(X_n)$ be an independent sequence of square-integrable random variables with means $\EE[X_n]=0$ and variances $\mathsf{Var}(X_n)=\sigma_n^2$ such that $\sum_{n\in\mathbb{N}}\sigma_n^2$ converges. Then $\sum_{n\in\mathbb{N}} X_n$ converges a.s.
\end{lemma}

Our second main ingredient now takes the following form:

\begin{lemma}\label{LipschitzSum}
Let $(\lambda_n)\subseteq (0,\infty)$ with $\sum_{n\in\mathbb{N}}\lambda_n=+\infty$ and $\sum_{n\in\mathbb{N}}\lambda_n^2<+\infty$. Assume $(\alpha_n)$ is a sequence of nonnegative square-integrable random variables which is i.i.d. Further, assume that $(\gamma_n)$ satisfies $\sup_{n\in\mathbb{N}} \gamma_n<+\infty$ a.s. Let $(\beta_n)$ be a sequence of nonnegative random variables such that $\sum_{n\in\mathbb{N}}\lambda_n\beta_n<+\infty$ a.s.\ and
\[
\beta_{n+1}-\beta_n \leq \theta \lambda_{n}\gamma_n\alpha_{n}\label{lipschitz}\tag{$+$}
\]
a.s.\ for all $n\in\mathbb{N}$, and some constant $\theta>0$. Then $\beta_n\to 0$ a.s.
\end{lemma}

This result can be considered to be a stochastic variant of a result due to Alber, Iusem and Solodov \cite{AlberIusemSolodov1998} (see Proposition 2 therein), by which for sequences of nonnegative real numbers $(\lambda_n), (\beta_n)$ which satisfy $\sum_{n\in\mathbb{N}}\lambda_n=+\infty$ and $\sum_{n\in\mathbb{N}}\lambda_n\beta_n<+\infty$ as well as
\[
\beta_{n+1}-\beta_{n}\leq \theta\lambda_n\text{ for all }n\in\mathbb{N}
\]
for some constant $\theta>0$, it holds that $\beta_n\to 0$ for $n\to\infty$. While our proof crucially uses many ideas from the associated proof of this result given in \cite{AlberIusemSolodov1998}, we were not able to fully reduce our stochastic Lemma \ref{LipschitzSum} to this deterministic result. We now turn to the proof itself.

\begin{proof}[Proof of Lemma \ref{LipschitzSum}]
Let $\underline{\alpha}$ be the mean of (each) $\alpha_n$. Consider the sequence of partial sums $\sum_{k=0}^n\lambda_k(\alpha_k - \underline{\alpha})$. Then $\EE[\lambda_n(\alpha_n - \underline{\alpha})]=0$ and the variance of $\lambda_n(\alpha_n - \underline{\alpha})$ is given by $\EE[\lambda_n^2(\alpha_n - \underline{\alpha})^2]=\lambda_n^2\EE[(\alpha_n - \underline{\alpha})^2]$. As each $\alpha_n$ is square-integrable and i.i.d., it holds that $\EE[(\alpha_n - \underline{\alpha})^2]<+\infty$ and the value is independent of $n\in\mathbb{N}$. Therefore $\sum_{n\in\mathbb{N}} \lambda_n^2\EE[(\alpha_n - \underline{\alpha})^2]<+\infty$ since $\sum_{n\in\mathbb{N}}\lambda_n^2<+\infty$. It follows by Kolmogorov's two-series theorem that $\sum_{k=0}^n\lambda_k(\alpha_k - \underline{\alpha})<+\infty$ a.s.

In particular, for any $\delta>0$, there exists an $N_\delta$ such that for all $n\geq m\geq N_\delta$, it holds that
\[
\left\vert \sum_{k=m}^n\lambda_k(\alpha_k-\underline{\alpha})\right\vert\leq\delta\text{ a.s.}
\]
Hence, in particular the following holds a.s., say on a set $\Omega_0$ of measure one: For any $\delta>0$, there exists an $N_\delta$ such that for all $n\geq m\geq N_\delta$:
\[
\sum_{k=m}^n\lambda_k\alpha_k\leq \underline{\alpha}\sum_{k=m}^n\lambda_k + \delta.\label{martingale}\tag{$-$}
\]

Suppose now that $\beta_n$ does not converge to $0$ a.s. Since we have $\sum_{n\in\mathbb{N}}\lambda_n\beta_n<+\infty$ a.s., as well as $\sum_{n\in\mathbb{N}}\lambda_n=+\infty$, we have $\liminf_{n\to\infty}\beta_n=0$ a.s., say on a set of measure one $\Omega_1$. By that however, we hence do not have $\limsup_{n\to\infty}\beta_n=0$ a.s. Correspondingly, this fails on a set of positive measure $\Omega_2$, say with $\PP(\Omega_2)\geq p_0$ where $p_0\in (0,1)$. Further, suppose that \eqref{lipschitz} holds on a set of measure one $\Omega_3$. Lastly, since $\sup_{n\in\mathbb{N}} \gamma_n<+\infty$ a.s., let $\psi:(0,1)\to (0,\infty)$ be a function such that $\PP\left(\sup_{n\in\mathbb{N}}\gamma_n>\psi(\lambda)\right)\leq\lambda$. In particular, we hence have $\PP\left(\sup_{n\in\mathbb{N}}\gamma_n\leq\psi(p_0/2)\right)\geq 1-p_0/2$, and we denote the inner set by $\Omega_4$.

Fix an $\omega\in\Omega_0\cap \Omega_1\cap\Omega_2\cap\Omega_3\cap \Omega_4$. As $\omega\in \Omega_2$, there is a number $\varepsilon$ such that $\limsup_{n\to\infty} \beta_n(\omega)>\varepsilon$, so choose a strictly increasing sequence $(n_i)$ of indices such that $\beta_{n_i}(\omega)\geq \varepsilon$ for all $i\in\mathbb{N}$. As $\omega\in \Omega_1$, we have $\liminf_{n\to\infty}\beta_n(\omega)=0$, there is a strictly increasing sequence $(m_i)$ such that $\beta_{m_i}(\omega)<\varepsilon/2$. Moreover, we can choose $m_i$ such that $m_i<n_i$ and $\beta_{n}(\omega)\geq \varepsilon/2$ for all $n$ such that $m_i<n\leq n_i$. Using $\omega\in \Omega_3$ and so \eqref{lipschitz} as well as $\omega\in\Omega_4$, we 
then have
\begin{align*}
\frac{\varepsilon}{2}&\leq \beta_{n_i}(\omega)-\beta_{m_i}(\omega)= \sum_{k=m_i}^{n_i-1} (\beta_{k+1}(\omega)-\beta_k(\omega))\\
&\leq \theta\sum_{k=m_i}^{n_i-1} \lambda_{k}\gamma_k(\omega)\alpha_{k}(\omega)\leq \theta\psi(p_0/2)\sum_{k=m_i}^{n_i-1} \lambda_{k}\alpha_{k}(\omega).
\end{align*}
For $i\in\mathbb{N}$ suitably large such that $n_i>m_i\geq N_\delta$ for $\delta=\varepsilon/(4\theta\psi(p_0/2))$, the estimate \eqref{martingale} yields
\[
\frac{\varepsilon}{2}\leq \theta\psi(p_0/2)\underline{\alpha}\sum_{k=m_i}^{n_i-1}\lambda_{k} + \frac{\varepsilon}{4}.
\]
Without loss of generality, we can now assume that $\underline{\alpha}>0$, as if $\underline{\alpha}=0$, then $\alpha_n=0$ for all $n\in\mathbb{N}$ so that $(\beta_n)$ would be decreasing by \eqref{lipschitz}, which combined with $\liminf_{n\to\infty}\beta_n=0$ would yield $\lim_{n\to\infty}\beta_n=0$ already. Thus, the above yields
\[
\sum_{k=m_i}^{n_i-1}\lambda_{k} \geq \frac{\varepsilon}{4\theta\psi(p_0/2)\underline{\alpha}},
\]
and so we have 
\begin{align*}
\sum_{k=m_i}^{n_i-1}\lambda_{k}\beta_{k}(\omega)&\geq \sum_{k=m_i+1}^{n_i-1}\lambda_{k}\beta_{k}(\omega)\geq\frac{\varepsilon}{2}\sum_{k=m_i+1}^{n_i-1}\lambda_{k}\\
&=\frac{\varepsilon}{2}\sum_{k=m_i}^{n_i-1}\lambda_{k}-\frac{\varepsilon}{2} \lambda_{m_i}\geq \frac{\varepsilon^2}{8\theta\psi(p_0/2)\underline{\alpha}}-\frac{\varepsilon}{2} \lambda_{m_i}.
\end{align*}
If $i\in\mathbb{N}$ is large enough, then $\lambda_{m_i}<\varepsilon/8\theta\psi(p_0/2)\underline{\alpha}$ since $\lambda_n\to 0$. Hence, for large enough $i\in\mathbb{N}$, we get
\[
\sum_{k=m_i}^{n_i-1}\lambda_{k}\beta_{k}(\omega)\geq \frac{\varepsilon^2}{16\theta\psi(p_0/2)\underline{\alpha}},
\]
so that summing over such large enough $i\in\mathbb{N}$ where the intervals spanning from $m_i$ to $n_{i}-1$ are disjoint, we get that $\sum_{n\in\mathbb{N}}\lambda_{n}\beta_{n}(\omega)=+\infty$, hence we have shown that $\sum_{n\in\mathbb{N}}\lambda_{n}\beta_{n}=+\infty$ on $\Omega':=\Omega_0\cap \Omega_1\cap\Omega_2\cap \Omega_3\cap\Omega_4$. The set $\Omega'$ has positive measure as by the Fr\'echet inequalities, we have
\[
\PP(\Omega')\geq \sum_{i=0}^4\PP(\Omega_i)-4=\PP(\Omega_2)+\PP(\Omega_4)-1\geq p_0-p_0/2=p_0/2>0.
\]
This contradicts that $\sum_{n\in\mathbb{N}}\lambda_n\beta_n<+\infty$ holds a.s.
\end{proof}

\subsection{Derivation of the main result}

We now turn to the weak convergence result for the method \eqref{RPPA}. As in the context of Theorem \ref{RPPAweakConv}, define
\[
F(x):=\int f(e,x)\,d\mu(e)
\]
and assume that $F$ is proper. By Fatou's lemma, since $f$ is a normal convex integrand, it follows that $F$ is also lsc. Further, $F$ retains the convexity of $f$. Thereby, $F$ is even weakly lower semicontinuous (weakly lsc), i.e.\
\[
\liminf_{n\to\infty} F(x_n)\geq F(x)
\]
whenever $(x_n)\subseteq X$ and $x\in X$ such that $x_n\to^w x$, which immediately follows from the following result of Ba\v{c}\'ak:

\begin{lemma}[Lemma 3.1 in \cite{Bacak2013}]\label{weakLsc}
Let $X$ be a Hadamard space. If $f:X\to (-\infty,+\infty]$ is lsc and convex, then it is weakly lsc.
\end{lemma}

In the following, we now assume that $F$ has a minimizer. We denote the set of minimizers by $\mathrm{argmin}F$ and the minimal value by $\mathrm{min}F$.

The key result on the iteration given by \eqref{RPPA} is then a quasi-Fej\'er-type inequality given in Lemma \ref{IneqMain}. In fact, this inequality is already derived in passing in \cite{Bacak2018}, but since the result only appears there in the context of a broad local compactness assumption we prove it here again for the benefit of the reader. Also, we give a slightly different argument than that given in \cite{Bacak2018}.

For that, we first require the following property of the proximal map, which is however immediate from its definition:

\begin{lemma}[see e.g.\ Lemma 2.2.23 in \cite{Bacak2014a}]\label{resLemma}
For any $\lambda>0$, $x,y\in X$ and $e\in E$:
\[
f(e,\mathrm{prox}^f_{\lambda}(e,x))-f(e,y)\leq\frac{1}{2\lambda} d^2(x,y)-\frac{1}{2\lambda}d^2(\mathrm{prox}^f_{\lambda}(e,x),y).
\]
\end{lemma}

The key quasi-Fej\'er-type inequality then takes the form of the following Lemma \ref{IneqMain}. For that, we in the following set $\mathsf{F}_n:=\sigma(\xi_{1},\dots,\xi_{n})$ and we abbreviate $\EE[\cdot\mid\mathsf{F}_n]$ by $\EE_n$. Further, we write $\underline{L}:=\int L^2\,d\mu<+\infty$.

\begin{lemma}[essentially Ba\v{c}\'ak \cite{Bacak2018}]\label{IneqMain}
For any $y\in X$, there exists a constant $C_{y,p}>0$ such that for all $n\in\mathbb{N}$:
\[
\EE_n[d^2(x_{n+1},y)]\leq (1+2C_{y,p}\lambda_n^2\underline{L})d^2(x_n,y)-2\lambda_n(F(x_n)-F(y))+2C_{y,p}\lambda_n^2\underline{L}.
\]
\end{lemma}
\begin{proof}
Given $n\in\mathbb{N}$ and $y\in X$, Lemma \ref{resLemma} implies
\[
d^2(x_{n+1},y)\leq d^2(x_n,y)-2\lambda_n[f(\xi_{n+1},x_{n+1})-f(\xi_{n+1},y)]
\]
and so we immediately have
\begin{align*}
\EE_n[d^2(x_{n+1},y)]&\leq d^2(x_n,y)-2\lambda_n\EE_n[f(\xi_{n+1},x_{n+1})-f(\xi_{n+1},y)]\\
&=d^2(x_n,y)-2\lambda_n\EE_n[f(\xi_{n+1},x_n)-f(\xi_{n+1},y)]\\
&\qquad+2\lambda_n\EE_n[f(\xi_{n+1},x_n)-f(\xi_{n+1},x_{n+1})]\\
&=d^2(x_n,y)-2\lambda_n[F(x_n)-F(y)]\\
&\qquad+2\lambda_n\EE_n[f(\xi_{n+1},x_n)-f(\xi_{n+1},x_{n+1})]
\end{align*}
where the third equality follows by independence of $\xi_{n+1}$ and $x_n$, as well as the fact that $\xi_{n+1}$ has distribution $\mu$. Now, note that Lemma \ref{resLemma} together with \eqref{growth} yield
\begin{align*}
d^2(x_{n+1},x_n)&\leq 2\lambda_n[f(\xi_{n+1},x_n)-f(\xi_{n+1},x_{n+1})]\\
&\leq 2\lambda_nL(\xi_{n+1})(1+d(x_n,p))d(x_n,x_{n+1})
\end{align*}
so that we have $d(x_{n+1},x_n)\leq 2\lambda_nL(\xi_{n+1})(1+d(x_n,p))$. Using \eqref{growth}, we further have
\begin{align*}
f(\xi_{n+1},x_n)-f(\xi_{n+1},x_{n+1})&\leq L(\xi_{n+1})(1+d(x_n,p))d(x_n,x_{n+1})\\
&\leq 2\lambda_nL^2(\xi_{n+1})(1+d(x_n,p))^2\\
&\leq 4\lambda_nL^2(\xi_{n+1})(1+d^2(x_n,p)).
\end{align*}
In particular, there is a constant $C_{y,p}>0$ such that
\[
f(\xi_{n+1},x_n)-f(\xi_{n+1},x_{n+1})\leq C_{y,p}\lambda_nL^2(\xi_{n+1})(1+d^2(x_n,y))
\]
so that we have
\[
2\lambda_n\EE_n[f(\xi_{n+1},x_n)-f(\xi_{n+1},x_{n+1})]\leq 2C_{y,p}\lambda_n^2\underline{L}(1+d^2(x_n,y)),
\]
again using the independence of $\xi_{n+1}$ to $\mathsf{F}_n$ and $x_n$. Combined, we get
\[
\EE_n[d^2(x_{n+1},y)]\leq (1+2C_{y,p}\lambda_n^2\underline{L})d^2(x_n,y)-2\lambda_n(F(x_n)-F(y))+2C_{y,p}\lambda_n^2\underline{L}
\]
as claimed.
\end{proof}

Combining the previous approach for weak convergence of the deterministic proximal point method in Hadamard spaces given in the seminal paper \cite{Bacak2013} with the Lemmas \ref{weakConvergence} and \ref{LipschitzSum}, we can now give the following proof of Theorem \ref{RPPAweakConv}:

\begin{proof}[Proof of Theorem \ref{RPPAweakConv}]
Using Lemma \ref{IneqMain} for some $z\in\mathrm{argmin}F$, we get 
\[
\EE_n[d^2(x_{n+1},z)]\leq (1+2C_{z,p}\lambda_n^2\underline{L})d^2(x_n,z)-2\lambda_n(F(x_n)-\mathrm{min}F)+2C_{z,p}\lambda_n^2\underline{L}.\label{centralFejer}\tag{$\circ$}
\]
By the assumptions on the parameters \eqref{parameters}, we get $\sum_{n\in\mathbb{N}}\lambda_n^2<+\infty$ so that the assumptions of Lemma \ref{weakConvergence}, in particular \eqref{randFejer}, are met. Item (1) of that lemma implies
\[
\sum_{n\in\mathbb{N}}\lambda_n[F(x_n)-\mathrm{min}F]<+\infty\text{ a.s.}
\]
By assumption \eqref{growth}, we have
\[
F(x)-F(y)\leq \int L(e)\,d\mu(e)(1+d(x,p))d(x,y),
\]
so that for $\beta_n:=F(x_n)-\min F$, we have
\[
\beta_{n+1}-\beta_n= F(x_{n+1})-F(x_n)\leq \int L(e)\,d\mu(e) (1+d(x_{n+1},p))d(x_{n+1},x_n).
\]
Further recall from the proof of Lemma \ref{IneqMain} that $d(x_{n+1},x_n)\leq 2\lambda_nL(\xi_{n+1})(1+d(x_n,p))$, so that we get 
\[
\beta_{n+1}-\beta_n\leq \int L(e)\,d\mu(e) 2\lambda_{n}(1+d(x_{n+1},p))(1+d(x_{n},p))L(\xi_{n+1})
\]
for all $n\in\mathbb{N}$. Note that by Lemma \ref{unifBd} together with inequality \eqref{centralFejer} above, we get that $\sup_{n\in\mathbb{N}} d(x_n,z)<+\infty$ a.s. Hence, for $\gamma_n:=(1+d(x_{n+1},p))(1+d(x_{n},p))$ we in particular have 
\[
\sup_{n\in\mathbb{N}} \gamma_n\leq \left(\sup_{n\in\mathbb{N}} (1+d(x_n,p))\right)^2\leq \left(\sup_{n\in\mathbb{N}} (1+d(x_n,z)+d(z,p))\right)^2<+\infty\text{ a.s.}
\]
So, also using \eqref{parameters}, the assumptions of Lemma \ref{LipschitzSum} are met with $\beta_n$ and $\gamma_n$ as above, as well as $\theta:=2\int L(e)\,d\mu(e)$ and $\alpha_n:=L(\xi_{n+1})$, so that we obtain $\beta_n\to 0$ a.s., that is
\[
F(x_n)\to\min F\text{ a.s.},
\]
say on a set $\widehat{\Omega}$ with measure one. 

We now show that on that set $\widehat{\Omega}$, we also have $\mathfrak{W}(x_n)\subseteq \mathrm{argmin}F$. For that, let $\omega\in \widehat{\Omega}$ and let $x(\omega)$ be a weak cluster point of $(x_n(\omega))$, i.e.\ $x_{n_k}(\omega)\to^w x(\omega)$ for some subsequence $(x_{n_k}(\omega))$. Since it follows from Lemma \ref{weakLsc} that $F$ is weakly lsc, as discussed before, and since $\omega\in\widehat{\Omega}$, we have
\[
F(x(\omega))\leq \liminf_{k\to\infty} F(x_{n_k}(\omega))=\lim_{n\to\infty} F(x_n(\omega))=\mathrm{min}F
\]
and so $x(\omega)\in\mathrm{argmin}F$. As we thus have $\mathfrak{W}(x_n)\subseteq \mathrm{argmin}F$ a.s., item (4) of Lemma \ref{weakConvergence} implies that $(x_n)$ a.s.\ weakly converges to an $\mathrm{argmin}F$-valued random variable.
\end{proof}

This solves Problem 6.6 of Ba\v{c}\'ak \cite{Bacak2023}. As a corollary, we obtain from Theorem \ref{RPPAweakConv} the almost sure weak convergence of the splitting proximal point method with random order given in \cite{Bacak2014b}, that is Theorem \ref{SplitPPAweak}, which in particular solves Problem 6.4 of Ba\v{c}\'ak \cite{Bacak2023}.

\begin{proof}[Proof of Theorem \ref{SplitPPAweak}]
Consider the space $(E,\mathcal{E},\mu)$, where $E=\{1,\dots,N\}$, $\mathcal{E}=2^E$ and $\mu$ is the uniform distribution. On that space, define $f:E\times X\to (-\infty,+\infty]$ by $f(k,x):=f_k(x)$. The result then immediately follows from Theorem \ref{RPPAweakConv} under the corresponding assumptions \eqref{parameters} and \eqref{growth}, where the latter translates to the condition that $f_k(x) - f_k(y)\leq L(k)(1 + d(x,p))d(x,y)$ for all $k = 1,\dots,N$ and $x,y \in X$, for some $L(k) > 0$ (recall footnote \ref{splitPPAfootnote}).
\end{proof}

\noindent\textbf{Acknowledgements:} The author wants to thank Miroslav Ba\v{c}\'ak and Thomas Powell for helpful conversations on the topic of this paper. Further, he particularly wants to thank Morenikeji Neri, not only for many insightful discussions and references, but also for proof reading various parts of this paper.

The author was originally stuck on the proof of a previous version of Lemma \ref{LipschitzSum}, having essentially rederived the key parts of Proposition 2 from the work of Alber, Iusem and Solodov \cite{AlberIusemSolodov1998} (which was unknown to him at the time, but was later pointed out by Morenikeji Neri, for which he is grateful), but unsure how to control the involved series in a stochastic context and unsure about the precise formulation in general. This previous version of Lemma \ref{LipschitzSum} and its argument were streamlined and finalised by the author only after a series of interactions with OpenAI's GPT-5.5 embedded into Microsoft Copilot, but then later modified and extended without it, especially after having learned of \cite{AlberIusemSolodov1998}. The author takes full accountability for the work. In particular, no other parts of this paper were worked on or impacted by the use of LLMs in any way. No parts of this paper were written or worded by an LLM.\\

\noindent\textbf{Conflicts of interests:} The author has no relevant financial or non-financial interests to disclose.\\

\noindent\textbf{Data availability:} Data sharing not applicable to this article as no data sets were generated or analyzed.

\bibliographystyle{plain}
\bibliography{ref}

\end{document}